\documentclass[11pt]{amsart}
\usepackage{url,bm,amssymb,amscd,amsmath}
\usepackage{graphics,epsfig,color}

\usepackage[letterpaper,asymmetric,left=1.25in,right=1.25in,top=1.1in,bottom=1.
25in,bindingoffset=0.0in]{geometry}

\usepackage[curve,color,graph]{xypic}

\renewcommand{\tilde}{\widetilde}
\renewcommand{\hat}{\widehat}

\newcommand{\hs}{X}

\newcommand{\sing}{{\rm sing}}
\newcommand{\reg}{{\rm reg}}

\newcommand{\bx}{{\bm x}}
\newcommand{\by}{{\bm y}}
\newcommand{\bz}{{\bm z}}
\newcommand{\origin}{{\bm 0}}

\newtheorem{theorem}{\bf Theorem}[section]
\newtheorem{proposition}[theorem]{\bf Proposition}

\newtheorem{lemma}[theorem]{\bf Lemma}

\newtheorem*{THMA}{Theorem A}
\newtheorem*{THMB}{Theorem B}

\theoremstyle{remark}
\newtheorem{remark}[theorem]{\bf Remark}

\begin{document}

\title{Pulling back singularities of codimension one objects}

\begin{author}[L. Giraldo]{Luis Giraldo}
\email{luis.giraldo@mat.ucm.es}
\address{
Dept. \'Algebra, Geometr\'{i}a y Topolog\'{\i}a and IMI \\
Facultad de Ciencias Matem\'aticas \\
Universidad Complutense de Madrid \\
Ciudad Universitaria \\
Plaza de Ciencias 3 \\
28040 Madrid\\
Spain
}
\end{author}

\begin{author}[R. K. W. Roeder]{Roland K. W. Roeder}
\email{rroeder@math.iupui.edu}
\address{ %
IUPUI Department of Mathematical Sciences\\
LD Building, Room 224Q\\
402 North Blackford Street\\
Indianapolis, Indiana 46202-3267\\
 United States }
\end{author}

\date{\today}

\begin{abstract}
We prove that the preimage of a germ of a singular analytic hypersurface under
a germ of a finite holomorphic map $g: (\mathbb{C}^n,\origin) \rightarrow
(\mathbb{C}^n,\origin)$ is again singular.  This provides a generalization of
previous results of this nature by Ebenfelt-Rothschild \cite{ER}, Lebl~\cite{LEBL}, 
and Denkowski \cite{DE}.  The same statement is proved for
pullbacks of singular codimension one holomorphic foliations.
\end{abstract}


\maketitle

\section{Introduction}

This note is devoted to proving the following two theorems:

\begin{THMA}
Let $g: (\mathbb{C}^n,{\bm 0}) \rightarrow (\mathbb{C}^n,{\bm 0})$ be a germ of a finite holomorphic map.  Let $\hs$ be a germ of an analytic hypersurface
through ${\bm 0}$ that is singular at ${\bm 0}$.  Then, the germ of a hypersurface $g^{-1}(\hs)$ is singular at ${\bm 0}$.
\end{THMA}

\begin{THMB} 
Let $g: (\mathbb{C}^n,{\bm 0}) \rightarrow (\mathbb{C}^n,{\bm 0})$ be a germ of a finite holomorphic map.  Let $\mathcal{F}$ be a germ of a codimension one
holomorphic foliation that is singular at ${\bm 0}$.  Then, the pullback foliation $g^* \mathcal{F}$ is singular at ${\bm 0}$.
\end{THMB}

\noindent
In fact, these two theorems are equivalent; see Proposition \ref{PROP:A_AND_B_EQUIVALENT} below.

It is important to note that Theorem A is about the ``reduced structure'' of
analytic hypersurfaces and about the ``set-theoretic'' preimage---no
multiplicities are taken into consideration.   This causes a subtlety that is
the main challenge in the proofs.  In fact, the equivalence between Theorems A
and B plays an important conceptual role\footnote{We use 
it in the proof of Proposition \ref{PROP:ZERO_DIMENSION_SING_SET}, which is the most
technical step of the proof of Theorem A.} in working with the set-theoretic
preimage of a hypersurface.  This is because the pullback of a singular codimension one
holomorphic foliation is defined in a way that ignores the multiplicities arising
from the critical hypersurfaces of $g$; see Section \ref{SUBSEC:FOLIATIONS}.

When $X$ is a germ of an analytic subvariety of $\mathbb{C}^n$ having arbitrary
dimension, a version of Theorem A has been proved by Ebenfelt-Rothschild
\cite[Theorem 2.1]{ER}, Lebl \cite[Section 4]{LEBL}, and Denkowski
\cite[Theorem 1.2]{DE} under the additional hypothesis that the
Jacobian ${\rm det}(Dg)$ does not vanish identically on $g^{-1}(X)$.  However,
this hypothesis has been removed in the case that ${\rm dim}(X) = 1$, see
\cite[Theorem 4.1]{ER} and \cite[Section 4]{LEBL}.  We refer the reader to the
aforementioned papers for further details.

Surprisingly, until now, the statement of Theorem A seems to be open
in the case that ${\rm det}(Dg)$ is allowed to vanish identically on $g^{-1}(X)$.

Note also that a version of Theorem B was proved in dimension $2$ by Kaschner,
P\'erez, and the second author of the present paper in \cite[Lemma 2.4]{KPR}.

In Section \ref{SEC:BACKGROUND} we give background on finite holomorphic maps,
analytic hypersurfaces, and singular codimension one holomorphic foliations.
Section \ref{SEC:EQUIV} is devoted to proving that Theorems A and B are
equivalent.   We conclude the proof of both theorems by presenting a
proof of Theorem A in Section \ref{SEC:PROOFS}.

\vspace{0.1in}
\noindent\textbf{Notational Conventions}:\\
Germs and their representatives will often be used interchangeably as a minor abuse of notation.
The notation $g : (U, 0)  \rightarrow (V, 0)$ when used for $g$
that is not a germ means that $g : U \rightarrow V$ is the map and $g(0) = 0$.

\vspace{0.1in}
\noindent\textbf{Acknowledgments}:\\
We thank Javier Fern\'{a}ndez de Bobadilla and Jenia Tevelev for helpful
conversations.  We thank the referee for doing a careful reading of our paper
and for several helpful comments that have helped us to improve the writing.
Part of this research was done while the first author was visiting the
Department of Mathematics of Indiana University at Bloomington, supported by
the ``Salvador de Madariaga'' grant no. PRX18/00404 funded by the Spanish
Government.  He wants to thank the Indiana University Math Department for their
hospitality.  He has also been partially supported by the research grants
MTM2015-63612-P and PGC2018-098321-B-I00, of the Spanish Government.  The
second author was supported by NSF grant DMS-1348589.

\section{Background}
\label{SEC:BACKGROUND}

\subsection{Finite Maps}
A holomorphic germ $g: (\mathbb{C}^n,\origin) \rightarrow (\mathbb{C}^n,\origin)$ is {\em finite} if
$g^{-1}(\origin) \cap U = \{\origin\}$ for a sufficiently small neighborhood $U$ of $\origin$ and some representative $g$.
Note that this implies that if $\by$ is sufficiently close to $\origin$ then $g^{-1}(\by) \cap U$ is a finite set.
The following lemma is well-known:

\begin{lemma}\label{LEM:LEMMA_PROPER_MAP}
Let $g: (\mathbb{C}^n,{\bm 0}) \rightarrow (\mathbb{C}^n,{\bm 0})$ be a germ of
a finite holomorphic map.  Then, there are arbitrarily small neighborhoods
$U,W$ of ${\bm 0}$ and a representative $g: (U,{\bm 0}) \rightarrow (W,{\bm
0})$ of the germ that is proper.
\end{lemma}

We refer the reader to \cite[Theorem 15.1.6]{RUDIN} for a proof of Lemma \ref{LEM:LEMMA_PROPER_MAP}.
We also refer the reader to \cite[p. 667-670]{GH} for various additional properties of finite maps.

\subsection{Hypersurfaces}
We refer the reader to the monograph by Milnor \cite{MILNOR} for
more details than will be presented here.
Let $U$ be an open neighborhood of~${\bm 0}$ in $\mathbb{C}^n$.  The term {\em
hypersurface in $U$} will refer to a potentially reducible analytic subset of
$U$, each of whose irreducible components has codimension one.  
Throughout this paper we will be interested in the ``reduced structure'' of a
hypersurface, thinking of it as an analytic set rather than a divisor.

Note that if $\hs$ is a reducible germ of a hypersurface at $\origin$ then
$\hs$ is singular at $\origin$.  Moreover, reducibility is preserved under
preimage by $g$.  Therefore, it suffices to prove Theorem A for irreducible
germs, and most of our attention will be focused on them.

Associated to any irreducible hypersurface $\hs$ in $U$ is a holomorphic
function $\psi: U \rightarrow \mathbb{C}$ that is prime in the ring of germs of
holomorphic functions  $\mathcal{O}_{\mathbb{C}^n,{\bm 0}}$ such that
\begin{align}\label{EQN:DEFINING_HS}
\hs = \{{\bf x} \in U \, : \, \psi({\bf x}) = 0\} \quad \mbox{and} \quad \mbox{$d\psi \not \equiv 0$ on 
$\hs$}.
\end{align}
The condition that $d\psi \not \equiv 0$ on $\hs$
follows from $\psi$ being a prime element of
$\mathcal{O}_{\mathbb{C}^n,{\bm 0}}$ and it ensures that $\psi$ vanishes to
order one on the smooth part of $\hs$.  In other words, the Cartier divisor $(\psi)$ defines the
reduced divisor associated to $\hs$.  

\begin{lemma}\label{LEM:SINGULAR_PTS_HYPERSURFACE}
Suppose $\hs$ is an irreducible hypersurface passing through ${\bm 0}$ that is
given by~(\ref{EQN:DEFINING_HS}).  If $\hs$ is
singular at ${\bm 0}$ then there is a smaller neighborhood $U' \subset U$ of
${\bm 0}$ in which the singular locus of $\hs$ is given by
\begin{align*}
\hs_\sing  = \{\bx \in U' \, : \, d\psi(\bx) = 0\},
\end{align*}
where $d$ denotes the exterior derivative.
\end{lemma}

\begin{proof}
We have, 
\begin{align*}
\hs_\sing  = \{\bx \in U \, : \, \psi(x) = 0 \mbox{  and  } d\psi(\bx) = 0\}.
\end{align*}
Let
\begin{align*}
S  = \{\bx \in U \, : \, d\psi(\bx) = 0\}.
\end{align*}
We can choose $U' \subset U$ to be sufficiently small so that each irreducible component
of $S \cap U'$ passes through $\origin$.
For the remainder of the proof we will work entirely in $U'$.  

It remains to show that $\psi|_S({\bx}) \equiv 0$.  This is trivial if $S = \{\origin\}$.  Otherwise, let $S_0$ be an irreducible
component of $S$.   The smooth locus $S_{0,\reg}$ is a submanifold of $U'$
that is dense in~$S_0$.  Any two points $p, q \in S_{0,\reg}$ can be
connected by a smooth path $\gamma: [0,1] \rightarrow S_{0,\reg}$ and 
\begin{align*}
\psi(q) - \psi(p) = \int_\gamma d\psi = 0.
\end{align*}
Therefore $\psi$ is constant on $S_{0,\reg}$ and hence on $S_0$.  As $S_0$
passes through $\origin$ and $\psi(\origin) = 0$, $\psi$ vanishes on $S_0$.
\end{proof}

\subsection{Singular codimension one holomorphic foliations}
\label{SUBSEC:FOLIATIONS}
We refer the reader to the survey paper by Cerveau \cite{CERVEAU} for more
background than we will present here.  A singular codimension one holomorphic
foliation $\mathcal{F}$ in a sufficiently small neighborhood $U$ of $\origin$
in $\mathbb{C}^n$ can be described by a holomorphic one form $\omega$ having the following
two properties:
\begin{itemize}
\item[(i)] $\omega \wedge d\omega = 0$ (integrability), and
\item[(ii)] $\omega_\sing := \{\bx \in U \, : \,  \omega(\bx) = 0\}$ has codimension $\geq 2$.
\end{itemize}
It is a consequence of the Frobenius Theorem that any $\bx \in U \setminus \omega_\sing$ has a neighborhood $N$
in which there is a system of holomorphic coordinates $\bz = (z_1,\ldots,z_n)$ on $N$ such that $\omega = u(\bz) dz_1$
with $u(\bz) \neq 0$ on $N$ (i.e.\ $u$ is a unit).  In these coordinates, the tangent planes to ${\rm Ker}(\omega)$ become
the ``vertical hyperplanes'' $\{z_1 = {\rm const}\}$.  They are interpreted as {\em local leaves} of $\mathcal{F}$ in the local coordinate
$\bz$.  By transporting them back to the original coordinate $\bm y$ and gluing, one obtains {\em global leaves} of $\mathcal{F}$
on $U \setminus \omega_\sing$.  One calls $\omega_\sing$ the singular locus of $\mathcal{F}$ and denotes
it by $\mathcal{F}_\sing$.

Occasionally, and only when there is no ambiguity, we will use the term ``foliation'' to mean ``singular codimension one holomorphic
foliation''.

Let $U, W \subset \mathbb{C}^n$ be open sets and let $g: U \rightarrow W$ be a finite holomorphic mapping.
Suppose $\mathcal{F}$ is a singular codimension one holomorphic foliation that is defined on $W$ by a holomorphic one form $\omega$ (satisfying (i) and (ii)).
Then, one defines the {\em pullback} $g^* \mathcal{F}$ on $U$ by a holomorphic one form $g^\# \omega$
on $U$ obtained by dividing $g^* \omega$ by a suitable holomorphic function, in order that the zero locus of
$g^\# \omega$ is of codimension two or larger.\footnote{The notation $g^\# \omega$ is non-standard and the reader
should note that it is only well-defined up to multiplication by a non-vanishing holomorphic function.}  Because
pullback $g^*$ respects wedge product and exterior derivatives, $g^* \omega$ still satisfies (i), and a simple
calculation shows that (i) continues to hold after rescaling $g^* \omega$ to obtain $g^\# \omega$.  Meanwhile, the rescaling
was done so that $g^\# \omega$ satisfies (ii). 

\begin{remark}
It is possible that $g^* \mathcal{F}$ be singular at $\bm x$ even though $\mathcal{F}$ is regular at $g(\bm x)$.  For example,
in dimension two, if $\omega = dy_1$ and $g(x_1,x_2) = (x_1^2-x_2^2,x_2)$, then 
\begin{align*}
g^* \omega = d(x_1^2-x_2^2) = 2x_1 dx_1-2x_2dx_2 = g^\# \omega,
\end{align*}
which vanishes at $(0,0)$.
\end{remark}

It is immediate from the definition that in a sufficiently small neighborhood $N$ of any regular point for $g$ the pullback
$g^*$ sends local leaves of $\mathcal{F}$ in $g(N)$ to local leaves of $g^* \mathcal{F}$ in~$N$.  The following
lemma justifies that this also happens more generally:

\begin{lemma}\label{LEM:LOCAL_LEAFS_PULLBACK}
Suppose that $\bx$ is a regular point for $g^* \mathcal{F}$ and $g(\bx)$ is a regular point for~$\mathcal{F}$.  Then, there
exists a neighborhood $N$ of $\bx$ such that if $L$ is the local leaf of $\mathcal{F}$  through $g(\bx)$  in $g(N)$
then $g^{-1}(L)$ is the local leaf of $g^* \mathcal{F}$ through $\bx$ in $N$.
\end{lemma}

\begin{proof}
Since $g^* \mathcal{F}$ is regular at $\bx$ and $\mathcal{F}$ is regular at $g(\bx)$ we can choose a connected neighborhood $N$
of $\bm x$ sufficiently small so that:
\begin{itemize}
\item[(a)]  there is a system of local coordinates $(z_1,\ldots,z_n)$ in $N$ in which
local leaves of $g^* \mathcal{F}$ are given by vertical hyperplanes $\{z_1 = {\rm const}\}$,  and
\item[(b)] there is a system of local coordinates $(w_1,\ldots,w_n)$
in $g(N)$ in which $\mathcal{F}$ is given by vertical hyperplanes $\{w_1 = {\rm const}\}$.  
\end{itemize}
Let $N_0 \subset N \setminus {\rm crit}(g)$
be an open set.  (Remark that since $g$ is finite it is an open mapping \cite[Theorem 15.1.6]{RUDIN} implying that ${\rm crit}(g)$ is a proper analytic subvariety of $N$.)  Then, $g$ sends local leaves of $g^* \mathcal{F}$ within $N_0$ to local leaves of $\mathcal{F}$ within $g(N_0)$.
Therefore, if we write
\begin{align*}
(w_1,\ldots,w_n) = g(z_1,\ldots,z_n) = (g_1(z_1,\ldots,z_n),\ldots,g_n(z_1,\ldots,z_n)),
\end{align*}
then within $N_0$ we have that $g_1(z_1,\ldots,z_n)$ is independent of $z_2,\ldots,z_n$, i.e.\ 
\begin{align*}
g_1(z_1,\ldots,z_n) \equiv g_1(z_1).
\end{align*}
Because $N$ is connected, this property carries over to $g$ within all of $N$.
In particular, $g$ sends every local leaf of $g^* \mathcal{F}$ in $N$ to some local leaf of $\mathcal{F}$ in $g(N)$.

We can then choose a smaller neighborhood $N' \subset N$ so that the local leaf of $g^* \mathcal{F}$ through $\bx$
is the only local leaf of $g^* \mathcal{F}$ within $N'$ whose image contains $g(\bx)$.
\end{proof}

\vspace{0.1in}

\section{Equivalence of Theorems A and B}
\label{SEC:EQUIV}

\begin{proposition}\label{PROP:A_AND_B_EQUIVALENT}
Theorem A holds if and only if Theorem B holds.
\end{proposition}

\begin{proof}
It suffices to show that the following are equivalent:
\begin{itemize}
\item[(a)]
There is a germ of an irreducible hypersurface $\hs$ that is singular at
${\bm 0}$ such that $g^{-1}(\hs)$ is smooth, and

\item[(b)]
There is a germ of a singular codimension one holomorphic foliation
$\mathcal{F}$ that is singular at 
${\bm 0}$
such that $g^{*}(\mathcal{F})$ is smooth.
\end{itemize}

Suppose (a) in order to prove (b).  
Suppose $W$ is a neighborhood of $\origin$ in $\mathbb{C}^n$ and 
$\hs$ is an irreducible hypersurface in $W$ given by
\begin{align*}
\hs = \{{\by} \in W \, : \, \psi({\by}) = 0\} \quad \mbox{and} \quad \mbox{$d\psi \not \equiv 0$ on $\hs$},
\end{align*}
with $\origin \in \hs_\sing$.  
Lemma \ref{LEM:SINGULAR_PTS_HYPERSURFACE} allows us to 
shrink to a smaller neighborhood $W' \subset W$ of $\origin$ so that 
\begin{align*}
\hs_{\rm sing} = \{{\bm y} \in W' \, : \, d\psi({\bm y}) = 0\}.
\end{align*}
Since $\hs_{\rm sing}$ has codimension two or larger in $\mathbb{C}^n$, 
the one form $d\psi$ defines a singular codimension one holomorphic foliation
$\mathcal{F}$ on $W'$ with ${\bm 0} \in \mathcal{F}_{\rm sing} = \hs_{\rm sing}$.

 By hypothesis, $g^{-1}(\hs)$ is smooth, so we can choose
holomorphic coordinates $(x_1,\ldots,x_n)$ in a neighborhood $U$ of ${\bm 0}$ so
that $g^{-1}(\hs) = \{x_1 = 0\}$.
Therefore $\psi \circ g({\bm x}) = x_1^m u({\bm x})$ for some
integer $m \geq 1$ and some unit (non-vanishing holomorphic function) $u({\bm x})$.  Hence, 
\begin{align}\label{EQN:SMOOTH_PULLBACK1}
g^* (d \psi) = d (g^* \psi) = d(\psi \circ g({\bm x})) = d(x_1^m u({\bm x})) = m x_1^{m-1} d x_1 u({\bm x}) + x_1^m du({\bm x}).
\end{align}
Recall that the pullback $g^* \mathcal{F}$ is defined in a neighborhood of
${\bm 0}$ by any holomorphic one form $g^\# (d \psi)$  obtained by rescaling
$g^* (d \psi)$ in order that the zero locus of $g^\# (d \psi)$ is of codimension two or
larger.  In this case, one can use 
\begin{align}\label{EQN:SMOOTH_PULLBACK2}
g^\# (d \psi) := \frac{1}{x_1^{m-1}} g^* (d \psi) = m \ u({\bm x}) d x_1 + x_1 du({\bm x}).
\end{align}
Note that $g^\# (d \psi)({\bm 0}) =  m \ u({\bm 0}) d x_1 \neq 0$, so that 
$g^* \mathcal{F}$ is smooth in some neighborhood of ${\bm 0}$, thus proving (b).

\vspace{0.1in}
We now prove that (b) implies (a).  Suppose $\mathcal{F}$ is a singular codimension one
holomorphorphic foliation that is singular at ${\bm 0}$ and that $g^* \mathcal{F}$ is smooth in some neighborhood of~${\bm 0}$.   
We will use the hypothesis that $g^* \mathcal{F}$ is smooth to prove that $\mathcal{F}$
is locally given by level hypersurfaces $\{\psi(\by) = {\rm const}\}$ for a suitable holomorphic function $\psi$.  Supposing
that $\psi({\bm 0}) = 0$, this will produce for us an irreducible  singular hypersurface $\hs = \psi^{-1}(0)$ whose
preimage under $g$ is smooth.

More specifically, suppose that $\mathcal{F}$ is given in some neighborhood
$W_0$ of ${\bm 0}$ by the holomorphic one form $\omega$ satisfying (i) and (ii) from the definition
of foliation.
Because of the hypothesis that $g^* \mathcal{F}$ is smooth in a neighborhood of ${\bm 0}$, 
we can do a local holomorphic change of variables so that in some neighborhood $U_0$ of ${\bm 0}$
we have that $g^* \mathcal{F}$ is given by $d x_1$.  In other words, $g^* \mathcal{F}$ is given by the family of vertical
hyperplanes of the form $\{x_1 = {\rm const}\}$ within~$U_0$.

By Lemma \ref{LEM:LEMMA_PROPER_MAP} we can choose neighborhoods $U \subset U_0$
and $W \subset W_0$ of the origin in domain and codomain, respectively, so that $g: U
\rightarrow W$ is a proper map and so that $g^{-1}({\bm 0}) \cap W = {\bm 0}$.
Let us also suppose that $U$ and $W$ are connected.
Let $\delta$ denote the topological degree of this mapping (number of preimages
of a point, when counted with multiplicity).  Let $\pi_1: U \rightarrow
\mathbb{C}$ be projection onto the first coordinate, $\pi_1({\bm x}) = x_1$.
Since $g: U \rightarrow W$ is proper, we can take the pushforward $\psi({\bm
y}) := g_* \pi_1({\bm y})$, which is a holomorphic function on $W$ defined by
\begin{align*}
\psi({\bm y}) := \left(g_* \pi_1\right)({\bm y})  = \sum_{{\bm x} \in g^{-1}({\bm y})} \pi_1({\bm x}),
\end{align*}
where preimages are counted with multiplicities.  (Remark that in certain contexts from
algebraic geometry the proper pushforward is called the ``trace map''; see, for example, \cite[p. 668]{GH}.)

We claim that level hypersurfaces of $\psi$ are tangent to leaves of $\mathcal{F}$ at those places where both
are regular.  
This will hold if and only if 
\begin{align}\label{EQN:WEDGE_CONDITION}
d \psi \wedge \omega = 0 \quad \mbox{on $W$},
\end{align}
 which is equivalent to the kernels of both $d \psi$ and $\omega$ being parallel.  
Because $d \psi \wedge \omega$ is a holomorphic two-form and $W$ is connected, it suffices to check
(\ref{EQN:WEDGE_CONDITION}) on any open subset of~$W$.

Let $N$ be a simply-connected open subset of $W$ that is disjoint from the critical
value locus, so that the Inverse Function Theorem and
Monodromy Theorem allow us to construct $\delta$ holomorphic
inverse branches $h_1,\ldots,h_\delta: N \rightarrow U$ of $g$.
Because $g^* \mathcal{F}$ is given by $d x_1$, we have that 
\begin{align*}
\pi_1 \circ h_\ell({\bm y}) \quad \mbox{is constant on $\gamma \cap N$}
\end{align*}
for any leaf $\gamma$ of $\mathcal{F}$ and any  $1 \leq \ell \leq \delta$.
This implies that $\psi$ is constant on $\gamma \cap N$ and thus that $d \psi
\wedge \omega = 0$ on the open set $N \subset W$.

It is possible from the definition of $\psi$ that $d \psi$ vanishes on some hypersurfaces.
Let $\eta$ be a holomorphic one-form obtained by dividing $d \psi$ by a suitable holomorphic function
in order that the zero locus of $\eta$ has codimension two or larger.

Let $\hs$ be the connected\footnote{A priori, $X$ could be a reducible hypersurface.  
We'll rule this out in the last paragraph of the proof. } component of 
\begin{align*}
\{\by \in W \, : \, \psi(\by) = 0\}
\end{align*}
that contains $\origin$.  We claim that $\hs$ is singular at $\origin$ and
$g^{-1}(\hs) = \{x_1 = 0\}$, so that $g^{-1}(\hs)$ is smooth.

Suppose for contradiction that $\hs$ is smooth at $\origin$.  Then, possibly after restricting to a smaller
neighborhood of $\origin$,  we can choose a suitable
local coordinate~$\bz$ in which $X = \{z_1 =~0\}$.  This implies that in this
coordinate $\psi(\bz) = z_1^n u(\bz)$ for some unit $u(\bz)$.  The same calculation as in
(\ref{EQN:SMOOTH_PULLBACK1}) and (\ref{EQN:SMOOTH_PULLBACK2}) can then be used
to show that $\eta(\origin) \neq 0$ and hence that $\mathcal{F}$ is smooth at
$\origin$, contrary to our hypothesis.

Meanwhile, the smooth locus $\hs_\reg$ is a leaf of $\mathcal{F}$.  Since $g^*
\mathcal{F}$ is given by the vertical leaves $\{x_1 = {\rm const}\}$, Lemma
\ref{LEM:LOCAL_LEAFS_PULLBACK} gives that each point $\by \in \hs_\reg$ has a
neighborhood in $\hs_\reg$ each of whose preimages under $g$ is contained in
one of the vertical leaves $\{x_1 = {\rm const}\}$.  However, since $\hs_\reg$
is dense in $\hs$ and since $g^{-1}(\origin) = \{\origin\}$ we conclude that
$g^{-1}(\hs) = \{x_1 = 0\}$.  Finally, remark that if $\hs$ were reducible,
then $g^{-1}(\hs)$ would also be reducible, contradicting the conclusion of the
previous sentence.

\end{proof}

\section{Proof of Theorems A and B}
\label{SEC:PROOFS}

In Proposition \ref{PROP:A_AND_B_EQUIVALENT} we saw that Theorems A and B are
equivalent.  Therefore, it suffices to prove Theorem A, and we will focus
entirely on that for the remainder of the paper.  
Let us start with two propositions.

\begin{proposition}\label{PROP:ZERO_DIMENSION_SING_SET}
Let $g: (\mathbb{C}^n,{\bm 0}) \rightarrow (\mathbb{C}^n,{\bm 0})$ be a germ of
a finite holomorphic map.  Let $\hs$ be a germ of an
analytic hypersurface that has an isolated singularity at ${\bm 0}$.  Then, the
preimage $g^{-1}(\hs)$ is singular at ${\bm 0}$.
\end{proposition}

\begin{remark}
A version of this was proved for $\mathbb{C}^2$ in \cite[Lemma 2.4]{KPR} and the proof
below extends what was done there to higher dimensions by using the algebraic technique of Koszul complexes.
\end{remark}

\begin{proof}
If $\hs$ is reducible at $\origin$ then $g^{-1}(\hs)$ is also, and hence $g^{-1}(\hs)$ is singular at $\origin$.

Now suppose for contradiction that $\hs$ is an irreducible analytic hypersurface defined in a
neighborhood $W$ of ${\bm 0}$ and singular only at ${\bm 0}$, with the property
that $g^{-1}(\hs)$ is smooth.  In this one step of the proof of Theorem A, it
will be convenient to use Proposition \ref{PROP:A_AND_B_EQUIVALENT} to replace
$\hs$ with a singular codimension one holomorphic foliation $\mathcal{F}$ that is singular
at ${\bm 0}$ whose pullback $g^{*} \mathcal{F}$ is smooth.  Moreover, in the
proof of Proposition \ref{PROP:A_AND_B_EQUIVALENT} we saw that the resulting
foliation satisfies $\mathcal{F}_{\rm sing} = \hs_{\rm sing}$, so that ${\bm 0}$
is an isolated singularity for $\mathcal{F}$.

Let us write $g$ in coordinates as $g(x_1,\ldots,x_n) = (g_1(x_1,\ldots,x_n),\ldots,g_n(x_1,\ldots,x_n))$.  Suppose $\mathcal{F}$ is given in a neighborhood of ${\bm 0}$ by the holomorphic
one form 
\begin{align*}
\omega = \sum_{i=1}^n a_i({\bm y}) dy_i.
\end{align*}
Remark that the assumption that $\mathcal{F}$ has an isolated singularity at ${\bm 0}$ implies that the common zero set of the $a_i({\bm y})$ is the origin,
in other words the $a_i({\bm y})$ for $1\leq i \leq n$ form a {\em regular sequence} in the ring of germs of
holomorphic functions $\mathcal{O}_{\mathbb{C}^n,{\bm 0}}$.

Suppose for contradiction that $g^* \mathcal{F}$ is smooth in a neighborhood of ${\bm 0}$.  Then, after a suitable change
of variables in the domain, we can suppose
that $g^* \mathcal{F}$ is given by $g^\# \omega = dx_1$.  In other words, if we write
\begin{align*}
g^* \omega = \sum_{\ell=1}^n b_\ell({\bm x}) dx_\ell,
\end{align*}
then 
\begin{align}\label{EQN:VANISHING}
b_\ell({\bm x}) = \sum_{i=1}^n a_i(g({\bm x})) \frac{\partial g_i}{\partial x_\ell} \equiv 0 \quad \mbox{for $2 \leq \ell \leq n$.}
\end{align}
(Remark that one could have $b_1({\bm x}) \not \equiv 1$ as pull-back of foliations allows one
to eliminate common factors.)  

Consider the holomorphic one form
\begin{align*}
\eta = \sum_{i=1}^n a_i(g({\bm x})) dx_i.
\end{align*}
We will now interpret (\ref{EQN:VANISHING}) in terms of the Koszul Complex obtained by wedging with
$\eta$:
\begin{align*}
0 \longrightarrow \mathcal{O}_{\mathbb{C}^n,{\bm 0}} \stackrel{\wedge \eta}{\longrightarrow} \Omega^1_{\mathbb{C}^n,{\bm 0}} \stackrel{\wedge \eta}{\longrightarrow}   \Omega^2_{\mathbb{C}^n,{\bm 0}}  \stackrel{\wedge \eta}{\longrightarrow} \cdots  \stackrel{\wedge \eta}{\longrightarrow}  \Omega^{n-2}_{\mathbb{C}^n,{\bm 0}} \stackrel{\wedge \eta}{\longrightarrow}  \Omega^{n-1}_{\mathbb{C}^n,{\bm 0}} \stackrel{\wedge \eta}{\longrightarrow}  \Omega^{n}_{\mathbb{C}^n,{\bm 0}} \longrightarrow 0.
\end{align*}
Since $g$ is a finite map, the  $a_i(g({\bm x}))$ for $1\leq i \leq n$ also
form a regular sequence in~$\mathcal{O}_{\mathbb{C}^n,{\bm 0}}$.  Therefore,
the homology of the Koszul Complex vanishes in the second to last place (at the
term $\Omega^{n-1}_{\mathbb{C}^n,{\bm 0}}$). See, for example, \cite[Theorem
A2.49]{Eisenbud}.

For each $2 \leq \ell \leq n$
consider $\tau_\ell \in \Omega^{n-1}_{\mathbb{C}^n,{\bm 0}}$ given by
\begin{align}\label{EQN:DEF_TAU}
\tau_\ell := \sum_{i=1}^n (-1)^{n-i} \frac{\partial g_i}{\partial x_\ell} dx_1 \wedge \cdots \wedge \widehat{dx_i} \wedge  \cdots \wedge dx_n.
\end{align}
Here, the hat on $\widehat{dx_i}$ denotes that the $dx_i$ has been omitted from the monomial.
Condition~(\ref{EQN:VANISHING}) implies that 
\begin{align*}
\tau_\ell \wedge \eta = 0.
\end{align*}
As the homology of the Koszul Complex vanishes in the second to last place, there exists $\alpha_\ell \in \Omega^{n-2}_{\mathbb{C}^n,{\bm 0}}$ with
\begin{align*}
\alpha_\ell \wedge \eta = \tau_\ell.
\end{align*}
Let us write
\begin{align*}
\alpha_\ell =  \sum_{1 \leq j < k \leq n} c_\ell^{j,k}({\bm x}) dx_1 \wedge \cdots \wedge \widehat{dx_j} \wedge \cdots \wedge \widehat{dx_k} \wedge  \cdots \wedge dx_n,
\end{align*}
where $c_\ell^{j,k} \in \mathcal{O}_{\mathbb{C}^n,{\bm 0}}$.  

Computing $\alpha_\ell \wedge \eta$ and comparing the coefficients with (\ref{EQN:DEF_TAU}) we find that
\begin{align}
\frac{\partial g_1}{\partial x_\ell} &= \sum_{k=2}^n (-1)^{n-k} c_\ell^{1,k}(\bx) a_k(g(\bx)), \label{EQN:COEFF_FROM_WEDGE_PRODUCT}  \\
\frac{\partial g_i}{\partial x_\ell} &= \sum_{j=1}^{i-1} (-1)^{n-1-j} c_\ell^{j,i}(\bx) a_j(g(\bx)) + \sum_{k=i+1}^{n} (-1)^{n-k} c_\ell^{i,k}(\bx) a_k(g(\bx)) \quad \mbox{for $2 \leq i \leq n-1$, and}
\nonumber \\
\frac{\partial g_n}{\partial x_\ell} &= \sum_{j=1}^{n-1} (-1)^{n-1-j} c_\ell^{j,n}(\bx) a_j(g(\bx)). \nonumber
\end{align}
All we will need from (\ref{EQN:COEFF_FROM_WEDGE_PRODUCT})  is that
for each $2 \leq \ell \leq n$ and $1 \leq i \leq n$ there exist some $e_\ell^{i,m} \in \mathcal{O}_{\mathbb{C}^n,{\bm 0}}$
for $1 \leq m \leq n$ such that
\begin{align}\label{EQN:PARTIALS}
\frac{\partial g_i}{\partial x_\ell} = \sum_{m = 1}^n e_\ell^{i,m}({\bm x}) a_m(g(\bx)).
\end{align}
We will obtain a contradiction from (\ref{EQN:PARTIALS}) and the naive idea
behind it will be that the left hand side of (\ref{EQN:PARTIALS}) should vanish
to lower order than the right hand side because of the partial derivative.  We
will now make this idea rigorous.

Because $g$ is a finite map we have that 
$g_m(0,x_2,\ldots,x_n) \not \equiv 0$ for all but at most one coordinate $m$.  Let $m_1,\ldots,m_p$ be precisely the
coordinates for which this holds (either $p=n-1$ or $p=n$).  For each $1 \leq q \leq p$ we can write
\begin{align*}
g_{m_q}({\bm x}) = x_1 g_{m_q}^1({\bm x}) + g_{m_q}^2(x_2,\ldots,x_m) 
\end{align*}
with $g_{m_q}^1 \in \mathcal{O}_{\mathbb{C}^n,{\bm 0}}$ and $g_{m_q}^2 \in
\mathcal{O}_{\mathbb{C}^{n-1},{\bm 0}}$ such that $g_{m_q}^2(x_2,\ldots,x_m) \not \equiv 0$.
Looking at the power series expansion for each $g_{m_q}^2(x_2,\ldots,x_m)$
there is some smallest total degree $d$ of all monomials appearing over all $1 \leq q \leq p$.  Let
$x_2^{d_2}\cdots x_n^{d_n}$ be one such monomial of this total degree.  It may
occur for multiple choices of~$q$ and, without loss of generality, we can
assume it occurs for $q=1$.  Because the mapping $(x_1,\ldots,x_n) \mapsto
(x_1^N,x_2,\ldots,x_n)$ preserves the foliation given by $dx_1 = 0$ we can
replace $g$ with $\hat g(x_1,\ldots,x_n) := g(x_1^N,x_2,\ldots,x_n)$ and all of
the above discussion also holds for $\hat g$.  In particular, we can choose $N
> d$ so that every component of $\hat g$ vanishes to order at
least $d$ and the component $\hat g_{m_1}(x_1,\ldots,x_n)$ still has the term
$x_2^{d_2}\cdots x_n^{d_n}$ appearing in its power series.

Let $x_r$ be a coordinate appearing with positive exponent in the monomial
$x_2^{d_2}\cdots x_n^{d_n}$.  Then, for $i=m_1$ and $\ell=r$ the left hand side
of (\ref{EQN:PARTIALS}) vanishes to order $d-1$ and the right hand side
vanishes to order at least $d$.  This gives the contradiction.
\end{proof}

\begin{proposition}\label{PROP:SLICING}
Let $U,W \subset \mathbb{C}^n$ be neighborhoods of $\origin$ and suppose
$g : (U,{\bm 0}) \rightarrow  (W,{\bm 0})$ is a
finite holomorphic map, $\hs \subset W$ is an irreducible
hypersurface that is singular at ${\bm 0}$ and has singular locus $\hs_{\rm
sing}$ of positive dimension (passing through the origin), and $g^{-1}(\hs)$ is
smooth.  

Then there exists open sets $U' \subset U$ and $W' \subset W$ with $g(U') = W'$ and a hyperplane $H$ in~$\mathbb{C}^n$ such that:
\begin{enumerate}
\item $\hs \cap H$ is a singular hypersurface within $H \cap W'$,
\item $g^{-1}(H)$ is smooth within $U'$, and
\item $g^{-1}(\hs \cap H)$ is a smooth hypersurface in $g^{-1}(H) \cap U'$.
\end{enumerate}
\end{proposition}

\begin{proof}
There is a holomorphic $\psi: W \rightarrow \mathbb{C}$ such that
\begin{align}\label{EQN:GENERIC_PROOF_DEF_EQN}
\hs = \{{\bm y} \in W \, : \, \psi({\bm y}) = 0\} \quad \mbox{and} \quad \mbox{$d\psi \not \equiv 0$ on
$\hs$}.
\end{align}
Let $S \subset \hs$ be an irreducible component of $\hs_{\rm sing}$ of maximal dimension 
\begin{align*}
1 \leq k:= {\rm dim}(S) \leq n-2.
\end{align*}
Let $V$ be an irreducible component of $g^{-1}(S)$.  Since $g$ is a finite map ${\rm dim}(V)~=~k$
and ${\rm dim}(g^{-1}(S_\sing)) < k$.  Therefore, $V_\reg \setminus g^{-1}(S_\sing)$ is an
open dense subset of $V$.  We can therefore find a point $p_0 \in V_\reg \setminus g^{-1}(S_\sing)$ and
an open neighborhood $U'$ of $p_0$ in $\mathbb{C}^n$ such that $U' \cap V$ is contained in $V_\reg$ and
$W':=g(U')$ satisfies that $W' \cap \hs_{\rm sing} \subset S_\reg$.
Remark that $W'$ is an open neighborhood of $g(p_0)$ because
$g$ is also an open mapping \cite[Theorem 15.1.6]{RUDIN}.

The remainder of the proof will take place within the neighborhoods $U'$ and $W'$.
By abuse of notation we will now refer to $g$ as the 
surjective finite holomorphic mapping $g:U' \rightarrow W'$, refer to $\hs$ as a singular hypersurface
in $W'$ whose singular locus $\hs_\sing = S$ is a smooth manifold of dimension $k$ in $W'$,
and refer to $V$ as a smooth manifold of dimension $k$ in $U'$ satisfying $g(V) \subset S$.

For any non-constant holomorphic curve $\gamma: \mathbb{D} \rightarrow U'$
we have that $Dg(\gamma(t)) \gamma'(t)~\neq~{\bm
0}$ for generic $t \in \mathbb{D}$, since otherwise $g(\gamma(\mathbb{D}))$ would be a single point.  Since
$V$ is a smooth manifold of dimension $k > 0$ we can suppose
$\gamma(\mathbb{D}) \subset V$ and that allows us to find a point $p_1 \in V$
and a tangent vector ${\bm v}_1 \in T_{p_1} V$ such that ${\bm w}_1:=Dg(p_1)
{\bm v}_1 \in T_{g(p_1)} S$ is non-zero.

Let $H_0$ be a hyperplane through $g(p_1)$ that is transverse to ${\bm w}_1$.
By definition, this implies that $H_0$ is transverse to $S$ and to the map $g$.  Let $L$ be
the line through $g(p_1)$ in the direction of~${\bm w}_1$.  Denote by $H_t$ the
one-parameter family of hyperplanes parallel to $H_0$, which we can
parameterize by the intersection point $t \in H_t \cap L$.  As transversality is
an open condition, we can reduce the neighborhoods $U'$ and $W'=g(U')$ so that
each member of $H_t$ is transverse to both $S$ and $g$.  Without loss of
generality, we can suppose the parameter $t$ varies over the unit disc $\mathbb{D}$.

We will now select the desired hyperplane $H$ satisfying Properties (1-3) from this family~$H_t$.
Since each $H_t$ is transverse to $g$ we have that $g^{-1}(H_t)$ is a smooth  submanifold of~$U'$
so that (2) holds for all choices of $t$.

We will now show that Property (1) holds for all parameters $t$ outside of
a zero measure subset of $\mathbb{D}$.  
Let $\tau: W' \rightarrow \mathbb{D}$ be the function that assigns to each ${\bm y}$
the value of the parameter~$t$ such that ${\bm y} \in H_t$.  
Sard's Theorem gives that the critical values of $\tau|_{\hs_\reg}$ form
a measure zero set $\mathbb{D}_0 \subset \mathbb{D}$.  For any $t \in \mathbb{D} \setminus \mathbb{D}_0$
we have that $H_t$ is transversal to $\hs_\reg$.  At each ${\bm y} \in \hs_\reg$ 
we have 
\begin{align*}
{\rm Ker}(d \psi({\bm y})) = T_{\bm y} \hs_\reg.
\end{align*}
This implies that if $t \in \mathbb{D} \setminus \mathbb{D}_0$ then
$d(\psi|_{H_t}) \neq 0$ at every point of $\hs_\reg \cap H_t$.

Meanwhile, $d\psi$ vanishes on $\hs_\sing$.  Since each $H_t$ intersects $S =
\hs_\sing$ transversally for each parameter $t$ and since $\hs_\sing$ has
codimension at least two in $\mathbb{C}^n$, we have that $H_t \cap \hs_\sing$
has codimension at least two in $H_t$.  We conclude that for $t \in \mathbb{D}
\setminus \mathbb{D}_0$ we have
\begin{align*}
\hs \cap H_t := \{{\bm y} \in H_t \, : \, \psi|_{H_t}({\bm y}) = 0\}
\end{align*}
with $d(\psi|_{H_t})$ vanishing only on the non-empty subset $H_t \cap S$ that has codimension two or higher within
$H_t$.  Therefore, if $t \in \mathbb{D} \setminus \mathbb{D}_0$ the hyperplane $H_t$ satisfies Property (1).

We will now show that Property (3) holds for all parameters $t$ outside of
another zero measure subset of $\mathbb{D}$.  Because $g^{-1}(H_t)$ is smooth
for each $t \in \mathbb{D}$ we have that the family $g^{-1}(H_t)$ forms a
smooth codimension one holomorphic foliation of $U'$.  Parameterizing the leaves using a local transversal in $U'$
we can again use Sard's Theorem to show that all but a zero measure subset of the leaves $g^{-1}(H_t)$
are transverse to the smooth manifold $g^{-1}(\hs)$.
Such leaves satisfy that $g^{-1}(\hs) \cap g^{-1}(H_t)$ is smooth.

We therefore find the desired hyperplane $H$ by choosing $t$ outside of the union of the two measure zero
subsets of $\mathbb{D}$ that were described in the previous three paragraphs.
\end{proof}

\begin{proof}[Proof of Theorem A]
The proof will be by induction on the dimension $n$ of the ambient space $\mathbb{C}^n$.
When $n = 2$, the singular locus of $\hs$ must be of dimension $0$ and Theorem A follows from Proposition~\ref{PROP:ZERO_DIMENSION_SING_SET}.

Now suppose that the statement of Theorem A holds when the ambient space has
dimension $n \geq 2$ in order to prove it for when the ambient space has
dimension $n+1$.  Let $\hs$ be a germ of a hypersurface at the origin in
$\mathbb{C}^{n+1}$ that is singular at ${\bm 0}$.  If ${\bm 0}$ is an isolated
singularity for $\hs$, then the result again follows from Proposition
\ref{PROP:ZERO_DIMENSION_SING_SET}.   Meanwhile, if $\hs$ is reducible, then the preimage
under any finite map is again reducible, and hence singular.

Otherwise, we suppose for contradiction that there does exist a finite
holomorphic mapping $g : (U,{\bm 0}) \rightarrow  (W,{\bm 0})$ with $U,W
\subset \mathbb{C}^{n+1}$ open and an irreducible hypersurface~$\hs$ with
positive dimensional singular locus $\hs_\sing$ passing through ${\bm 0}$ such that
$g^{-1}(\hs)$ is smooth in~$U$.   In this case, Proposition~\ref{PROP:SLICING}
will allow us to reduce the dimension of the ambient space by one, in order
to contradict the induction hypothesis.

More specifically, Proposition~\ref{PROP:SLICING} gives us open sets $U' \subset U$ and $W' \subset W$ 
such that $g(U')=W'$ and a hyperplane 
$H$ such that $\tilde{\hs}:=\hs \cap H$ is again
singular in $H \cap W'$  but with $g^{-1}(H)$ smooth in $U'$ and $g^{-1}(\tilde{\hs})$
smooth within $g^{-1}(H) \cap U'$.  
Choosing local coordinates on
$g^{-1}(H)$ and~$H$, respectively, we obtain a new germ of a finite
holomorphic mapping $\tilde{g}:(\mathbb{C}^{n},{\bm 0})\rightarrow (\mathbb{C}^{n},{\bm 0})$ such that the preimage of a singular hypersurface
$\tilde{\hs}$ is smooth.  This violates the induction hypothesis.

We conclude that the statement of Theorem A holds for any dimension ambient space.

\end{proof}

\newpage

\end{document}